\documentclass[11pt]{amsart}

\usepackage{amsmath}
\usepackage{amssymb}
\usepackage{amsfonts}
\usepackage{amsthm}
\usepackage{enumerate}
\usepackage{bbm}
\usepackage[mathscr]{eucal}

\usepackage{graphicx}
\usepackage{verbatim}

\newcommand{\ci}[1]{_{{}_{\scriptstyle{#1}}}}

\newcommand{\Be}{\begin{equation}}
\newcommand{\Ee}{\end{equation}}

\newcommand{\Bm}{\begin{multline}}
\newcommand{\Em}{\end{multline}}
\newcommand{\Bea}{\begin{eqnarray}}
\newcommand{\Eea}{\end{eqnarray}}
\newcommand{\Beas}{\begin{eqnarray*}}
\newcommand{\Eeas}{\end{eqnarray*}}
\newcommand{\Benu}{\begin{enumerate}}
\newcommand{\Eenu}{\end{enumerate}}
\newcommand{\Bi}{\begin{itemize}}
\newcommand{\Ei}{\end{itemize}}

\def\tu{\widetilde u}
\def\tom{\widetilde \omega}

\def\intslash{\rlap{\kern  .32em $\mspace {.5mu}\backslash$ }\int}
\def\qsl{{\rlap{\kern  .32em $\mspace {.5mu}\backslash$ }\int_{Q_x}}}

\def\vth{\vartheta}

\def\emph#1{{\it #1 }}

\def\ga{\gamma}

\def\cf{{\it cf}}

\def\rank{{\text{\rm rank }}}

\def\supp{{\text{\rm supp}}}
\def\rad{{\text{\rm rad}}}

\def\inn#1#2{\langle#1,#2\rangle}

\def\noi{\noindent}

\def\lc{\lesssim}
\def\gc{\gtrsim}

\def\eps{\varepsilon}

\def\ka{\kappa}
            
\def\la{\lambda}

\def\vphi{\varphi}
             
\def\om{\omega}              \def\Om{\Omega}

\def\fv{{\mathfrak {v}}}

\def\bbR{{\mathbb {R}}}

\def\cB{{\mathcal {B}}}
\def\cC{{\mathcal {C}}}

\def\cF{{\mathcal {F}}}

\def\cL{{\mathcal {L}}}

\def\cR{{\mathcal {R}}}
\def\cS{{\mathcal {S}}}
\def\cT{{\mathcal {T}}}

\def\cV{{\mathcal {V}}}
\def\cW{{\mathcal {W}}}

 at 10 true pt

\def\be#1{\begin{equation}\label{ #1}}
\def\endeq{\end{equation}}
\def\endal{\end{align}}
\def\bas{\begin{align*}}
\def\eas{\end{align*}}
\def\bi{\begin{itemize}}
\def\ei{\end{itemize}}

\def\eps{\varepsilon}

\def\emph#1{{\it #1}}
\def\textbf#1{{\bf #1}}

\theoremstyle{plain}
   \newtheorem{theorem}{Theorem}[section]

   \newtheorem{proposition}[theorem]{Proposition}

   \newtheorem{corollary}[theorem]{Corollary}

   \newtheorem{theorem*}{Theorem}

\theoremstyle{remark}
   \newtheorem{remark}[theorem]{Remark}

\theoremstyle{definition}

\numberwithin{equation}{section}

\begin{document}

\title{Extensions of the Stein-Tomas theorem}

\author[J. Bak \ \ \ A. Seeger]{Jong-Guk Bak  \ \ \  \ \ Andreas Seeger}

\address {J. Bak \\ Department of Mathematics\\ Pohang University of Science and Technology
\\
Pohang 790-784, Korea}
\email{bak@postech.ac.kr}

\address{A. Seeger\\
Department of Mathematics\\ University of Wisconsin-Madison\\Madison, WI 53706, USA}
\email{seeger@math.wisc.edu}


\begin{thanks} {J.B. supported in part by Priority Research Centers Program through the National Research Foundation of Korea funded by the Ministry of
Education, Science and Technology (2009-0094068) and a Postech-BSRI grant.
A.S. supported in part by National Science Foundation grant 0652890.}
\end{thanks}

\begin{date}{April 27, 2010.} \end{date}
\begin{abstract}
We prove an
endpoint version of  the Stein-Tomas restriction theorem, for a  general
class of  measures, and   with a strengthened  Lorentz
space estimate.
A similar  improvement is obtained
for Stein's estimate  on  oscillatory integrals of
Carleson-Sj\"olin-H\"ormander type and some spectral projection operators on compact manifolds, and
 for classes of oscillatory integral
operators with one-sided  fold singularities.
\end{abstract}
\maketitle
\section{Introduction and statement of results}
\label{intro}

{\bf Fourier restriction.} Our first result concerns an endpoint  version of the $L^2$
Stein-Tomas Fourier
restriction theorem (\cite{tomas1}, \cite{tomas-stein}, \cite{stein}),
in the following general setup as formulated by Mockenhaupt \cite{mock}, and also by
Mitsis \cite{mitsis}.

Let $0<a<d$, $0<b\le a/2$, and consider a
 probability
measure  $\mu$ on $\bbR^d$.
We assume that, for positive finite constants $A\ge 1$, $B\ge 1$,
  $\mu$  satisfies
\begin{equation}\label{dimhyp}
\sup_{\rad(\cB)\le 1} \frac{ \mu(\cB)} {\rad(\cB)^a} \le A
\end{equation}
where the supremum  is taken over all balls $\cB$ with radius $\le 1$
 \begin{equation} \label{Fmuhyp}
\sup_{|\xi|\ge 1} |\xi|^b |\widehat \mu(\xi)|\le B.
\end{equation}
The number  $\inf \{a:\text{\eqref{dimhyp} holds for some $A<\infty$}\}$ is often referred to as  the
\lq  dimension\rq  \,of $\mu$ and the number
$\inf \{2b: \text{\eqref{Fmuhyp} holds for some $B<\infty$}\}$ is the
\lq Fourier dimension\rq \, of $\mu$.

The Stein-Tomas theorem (originally for surface measure on the sphere)
is concerned with  $L^p(\bbR^d, dx )\to L^2(d\mu)$ estimates for the
Fourier transform. Stein, in the 1960's, proved  that  such estimates
hold for some  $p>1$ if \eqref{Fmuhyp} holds for
some $b>0$. Tomas \cite{tomas1} improved  Stein's estimate
and obtained an almost sharp range. His proof was used  in
\cite{mock}, \cite{mitsis},
to show that, given \eqref{dimhyp} and \eqref{Fmuhyp},
\Be\label{mockest} \cF: L^p (dx)\to L^2(d\mu), \qquad  1\le p< p_\circ(a,b):=
\frac{2(d-a+b)}{2(d-a)+b}\,.
\Ee
 For surface measure on
 hypersurfaces with nonvanishing curvature one has  $a=2b=d-1$,
which gives the familiar parameter $p_\circ = \frac{2(d+1)}{d+3}$.
The  article  \cite{mock} was primarily  concerned with measures
 on Salem sets, i.e.
singular measures supported on $a$-dimensional subsets  of the real line
which satisfy \eqref{Fmuhyp}
 for  $b <a/2$ (with the parameter  $B$ depending on $b$).

Stein (\cf. \cite{tomas-stein}) proved an endpoint $L^p\to L^2(d\mu)$
estimate for the
surface measure on a sphere,
using interpolation with an analytic family of kernels.
As shown by Greenleaf \cite{greenleaf} this approach
can also be used  when  $\mu$ is surface measure on an imbedded
submanifold of $\bbR^d$, in order
 to get the endpoint bound for $p=p_\circ(a,b)$.
However it is not clear how to  extend the analytic  interpolation argument (and neither the alternative  interpolation argument in \cite{gv},
\cite{MSS})
to the general class  of measures
satisfying \eqref{dimhyp}, \eqref{Fmuhyp}.
 Here we establish the
endpoint version  of \eqref{mockest}
   and further strengthen  it by
 replacing   $L^{p_\circ}$
 with  the larger and generally  optimal  Lorentz space
$L^{p_\circ,2}$.
\begin{theorem}\label{lor}
Let $0<a<d$, $0<b\le a/2$,  and let $\mu $ be a probability
 measure satisfying  \eqref{dimhyp},
\eqref{Fmuhyp} with constants $A\ge 1$, $B\ge 1$.
Let  $p_\circ=\frac{2(d-a+b)}{2(d-a)+b}$.
 Then
\Be \label{lorestimate}
\int |\widehat f|^2 d\mu \le \,C^2
A^{\frac{b}{d-a+b}}  B^{\frac{d-a}{d-a+b}}\,
\|f\|_{L^{p_\circ,2}(\bbR^d)}^2\,.
\Ee
If $a,b$  are chosen from a compact interval $I\subset (0,\infty)$
then the constant  $C$ depends only on $d$ and $I$.
\end{theorem}
The proof of  Theorem \ref{lor} is given in  \S\ref{restrsect}.

\medskip



\noi {\it Remarks.}
(i)
By interpolation with the trivial $L^1\to L^\infty$ bound
we see that \eqref{lorestimate} implies
\Be \label{lorestimateq}
\Big(\int |\widehat f|^q d\mu\Big)^{1/q} \le \,C^{\frac{2}{q}}
A^{\frac{b}{(d-a+b)q}}  B^{\frac{d-a}{(d-a+b)q}}\,
\|f\|_{L^{p}(\bbR^d)}\,
\Ee
for $1\le p\le p_\circ(a,b)$, $q=\frac{b}{d-a+b} p'$. The dependence of the constant on $A$ and $B$ for $p< p_\circ(a,b)$
 has been  relevant in the work by \L aba and Pramanik \cite{lp}.

  By real interpolation Theorem \ref{lor} also  implies $\cF:L^{p,s}(dx)\to
L^{q}(d\mu)$ for
$1<p<p_\circ(a,b)$, $q\le \frac{b}{d-a+b}\, p'$, $s\le q$.
Here $p'=\frac{p}{p-1}$, the conjugate exponent.
In some  instances (e.g. \cite{z}, \cite{H},
 \cite{tvv})
the $L^p\to L^q$ result  for the
critical  $q= \frac{b}{d-a+b}\, p'$ is  known  even for some $p>p_\circ(a,b)$ and in such cases
 the Lorentz improvement of Theorem \ref{lor} for $q=2$ is of course trivial
 by interpolation.

(ii) Our estimates follow from off-diagonal $L^p\to L^q$ bounds for the
convolution operator with kernel $\widehat \mu$. These are known for the surface measure on  spheres, in particular   for  this example the
restricted weak type estimate in Proposition \ref{offdiag} below
is a special case of S. Guti\'errez' result \cite{gu}
on Bochner-Riesz operators with negative index.
Related  off-diagonal estimates are also featured
in \cite{bmo} where complex interpolation is used (and which contains also
several earlier references),  and,
more recently,   in the article
 \cite{kt} by Keel and Tao
 where real interpolation for bilinear operators is used to obtain
endpoint
$L^2_t(L^{p,2}_x)$ Strichartz estimates
(with the Lorentz norms in the slices).

\medskip

\noi{\it Sharpness of the Lorentz exponent.}  We consider the case of surface measure
on  the sphere and show
that for this example the Lorentz exponent in Theorem \ref{lor} is optimal.
Indeed we show that
$\cF$ does not map $L^{p,s}\to L^q$ for   $q= \frac{d-1}{d+1} p'$ and
$s>q$.
This
is seen by a superposition of  standard Knapp examples at different
scales.
Namely  let
$$g(\xi',\xi_d)= \sum_{k=1}^N 2^{k(d-1)/q}\eta_1(2^k |\xi'|)
\eta_0(2^{2k-5} (|\xi_d-1|))$$ where $\eta_1$, $\eta_0$ are
suitable bump functions on  $(\frac 34, \frac 54)$ and  $(-1,1)$,
respectively. Then
$(\int_{S^{d-1}}|g|^q d\sigma)^{1/q} \approx N^{1/q}$.
Also  $f=\cF^{-1}[g]$ is bounded  and the measure of the set
 $\{x:|f(x)|>2^{-j}\}$ is bounded by
$C 2^{k_j(d+1)}$ where $(d+1-\frac{d-1}q)k_j \in[j-c,j+c]$. Hence  if $k_j<N$
the measure of this set   is bounded  by
$C' 2^{j\frac{(d+1)q}{(d+1)q-(d-1)}}=C' 2^{jp}$
and if $k_j\ge N$ it is  $\lc 2^{Np}$.
Thus the $L^{p, s}$ norm  of
$f$
is $O(N^{1/s})$ and if
the
Fourier restriction operator  maps
 $L^{p,s}$ to $ L^q(d\mu)$ then
$s\le q$.

\medskip

\noi{\bf Operators of
 Carleson-Sj\"olin-H\"ormander type}. We consider oscillatory integral
operators $T_\la$
given by
\Be\label{Tladef}T_\la f(x) = \int \zeta(x,y) e^{i\la\varphi(x,y)} f(y) \, dy;\Ee
here  $\la> 1$, $\zeta\in C^\infty_c(\Omega_L\times\Omega_R)$ where $\Omega_L$ is an
open set in $\bbR^{d}$ and  $\Omega_R$ is an open set in $\Bbb
R^{d-1}$. The phase is real-valued and smooth on $\Om:=\Om_L\times \Om_R$
and the following conditions are assumed.

First, the mixed Hessian $\varphi_{xy}''$ has maximal rank
\Be\label{rank}  \text{rank}\, \varphi_{xy}''=d-1\Ee on
$\Om$.
This implies that for every $x\in \Om_L$ the  variety
\Be\label{Sigmax} \Sigma_x:=\{\varphi_x'(x,y): y\in \Om_R\}
\Ee is an immersed hypersurface in
$(\bbR^d)^*$.
The second hypothesis is then that for every $x\in \Omega_L$ the hypersurface
$\Sigma_x$ has nonvanishing Gaussian
  curvature everywhere. Analytically this means that
for any unit vector $u=(u_1,\dots, u_d)$
 we have the condition
\Be\label{CS}
u^{\text t}\varphi_{x y}=0 \,\implies\,
\det \Big( \nabla^2_{yy} (u^\text{t} \varphi_{x})\Big)\neq 0 \, ,
\Ee
for all points in $\Om$.

In \cite{H} H\"ormander raised the question  whether
conditions \eqref{rank}, \eqref{CS} imply
\Be\label{Tlaest}
\|T_\la\|_{L^p(\bbR^{d-1})\to  L^q(\bbR^d)} \lc \la^{-d/q},\quad
q=\frac{d+1}{d-1}p',
\Ee
for  $1<p<\frac{2d}{d-1}$. As he pointed out a limiting argument yields the analogous estimate
for the adjoint of the Fourier restriction operator; the relevant phase function is  $\varphi(x,y)=\inn{x}{\Gamma(y)}$ where $\Gamma$ parametrizes a hypersurface with nonvanishing curvature.
The optimal result in two dimensions was  proved  in \cite{H}
following earlier results by Fefferman and Stein \cite{feff} and by
Carleson and Sj\"olin \cite{casj}.
 Bourgain \cite{bo2} showed that in dimension $d\ge 3$
 there are classes of  phase-functions satisfying  \eqref{rank}, \eqref{CS}
for which
\eqref{Tlaest}  fails for any $p>2$.
Earlier,  Stein \cite{stein} had established \eqref{Tlaest} in the range
 $1\le p\le 2$. Here we are concerned with a Lorentz space
strengthening
of \eqref{Tlaest} for the endpoint $p=2$ of Stein's result, with  $L^q$ replaced by  $L^{q,2}$.

Following \cite{MSS} we slightly generalize the setup of  Stein's theorem
and relax the curvature assumptions on the
 manifolds $\Sigma_x$ in \eqref{Sigmax}, namely, we assume
that for every point on $\Sigma_x$ at least $\ka$ principal curvatures
do not vanish.
This is equivalent to
\Be\label{CSgen}
u^{\text t}\varphi_{x y}=0 \,\implies\,
\rank \Big( \nabla^2_{yy} (u^\text{t} \varphi_{x})\Big)\ge \kappa \, ,
\Ee
for all points in $\Om$,  and all unit vectors $u$. The case $\kappa=d-1$ corresponds to the setup described above and
the case $\kappa= d-2$ occurs in problems with
conical structures.

\begin{theorem}\label{oscthm}
Let
$T_\la$ be as in \eqref{Tladef}, with $\varphi$ satisfying
\eqref{rank}, \eqref{CSgen}. Let
$q_\circ=2+4\kappa^{-1}$. Then
$$\|T_\la\|_{L^2(\bbR^{d-1})\to L^{q_\circ,2}(\bbR^d)}\lc
\la^{-d/q_\circ}.
$$
\end{theorem}
We give the proof of   Theorem \ref{oscthm}
in \S\ref{oscintsect}.

\medskip

\noi{\bf Spectral projection operators on compact manifolds.}
As an application we mention a slight  improvement
of the $L^2\to
L^q$  endpoint bounds for
spectral projection operators associated to  the Laplace-Beltrami
operator
on
general  compact Riemannian
manifolds, due to Sogge \cite{sogge2}. See also
\cite{seso} for  a result covering first order pseudo-differential
operators and then  some higher order differential  operators.

Following the latter paper, and \cite{soggebook}, we consider a classical
elliptic pseudo-differential
operator of first order on a $d$-dimensional  compact manifold $M$ which is self-adjoint
with respect to some given density. We denote by $p(x,\xi)$ the
principal symbol, which is homogeneous of degree one with respect to
$\xi$, and only vanishes for $\xi=0$. Our hypothesis is that the co-spheres
$$\Sigma_x=\{\xi: p(x,\xi)=1\}$$
are convex, with non-vanishing Gaussian curvature everywhere (this
property is referred to as ``strict convexity'' in \cite{seso}). Of
course the main example is given by $P=\sqrt{-\Delta}$ where $\Delta$
is the  Laplace-Beltrami operator on $M$. Consider the
finite dimensional space of eigenfunctions of $P$ whose eigenvalues
belong to $[\la,\la+1]$, for $\la\gg1$, and  the self-adjoint
projection to this finite-dimensional subspace. We denote this projection operator
$\chi_\la(P)$ (where $\chi_\la$ is the characteristic function of
$[\la,\la+1]$).
By the results in \cite{sogge2}, \cite{seso} the $L^2(M)\to L^q(M)$
operator norm of $\chi_\la(P)$ is $O(\la^{d(1/2-1/q)-1/2})$ in the
  sharp range $q_\circ:=\frac{2d+2}{d-1}\le q\le \infty$; in
  particular one has the bound $O(\la^{1/q_\circ})$ for
  $q=q_\circ$. The argument in  \cite{seso} relies on the small time parametrix construction for solutions of the wave equation in \cite{Hsp}, and so
does the  treatment in ch.5 of \cite{soggebook}. In the latter the  $L^2\to L^q$ estimates for $\chi_\la(P)$
are directly reduced   to $L^2\to L^q$ inequalities for
oscillatory integral operators of
Carleson-Sj\"olin-H\"ormander type. Thus using  this approach Theorem \ref{oscthm} can be used to derive
the following  endpoint  result.

\begin{corollary}\label{spectrthm} For $\la\ge 1$, $q_\circ= \frac{2d+2}{d-1}$, the operators
$\la^{-1/q_\circ}\chi_\la(P)$  map $L^2(M)$ to $L^{q_\circ,2}(M)$ and $L^{q_\circ',2}(M)$ to $L^2(M)$, with operator norms uniform in $\la$.
\end{corollary}

\medskip

\noi {\bf Operators with one-sided fold singularities.}
One can also prove   Lorentz-space improvements
of the endpoint $L^2(\bbR^d)\to L^q(\bbR^d)$ results
for oscillatory integral operator with one-sided fold singularities,
obtained by Greenleaf and the second author in \cite{grse}.
Here one considers the operator defined by
\Be\label{Tlaoscdef}
\cT_\la f(x) = \int \zeta(x,y) e^{i\la\Phi(x,y)} f(y) \, dy,\Ee
where  $\zeta\in C^\infty_c(\Omega_1\times\Omega_2)$ and now
$\Omega_1, \Omega_2\subset \bbR^d$.
The phase  $\Phi$ is  smooth and  real-valued in $\Omega_1\times \Omega_2$
and $T_\la$ now acts on functions of $d$ variables.
We assume that the map  $$\pi_L : (x,y)\mapsto (x,\Phi_x(x,y))$$ has only fold singularities in $\Omega_1\times\Omega_2$,
i.e.
\Be\label{foldhyp}
\begin{gathered}
\text{rank}\, \Phi_{xy}''\ge d-1
\\
\det (\Phi_{xy})(x,y) =0, \, \Phi_{xy}b=0,\, b\neq 0 \quad \implies
\inn{b}{\nabla_y} (\det\Phi_{xy}) \neq 0.
\end{gathered}
\Ee
For an integer $\kappa$, $0\le \ka\le d-1$ we say that
{\it Hypothesis $(\pi_L, \kappa)$} is satisfied
if the $(d-1)$-dimensional immersed
hypersurfaces
$$\cL_x=\{ \Phi_x(x,y): \det \Phi_{xy}=0 \}$$
have at least $\kappa$ nonvanishing principal curvatures at every point.
Notice that  the case $\kappa=0$  is included (and contains no
particular  assumption).

\begin{theorem}\label{oscfoldthm}
Let
$\cT_\la$ be as in \eqref{Tlaoscdef}, with $\Phi$ satisfying
\eqref{foldhyp} and also Hypothesis $(\pi_L, \kappa)$.
Let
$q_1=\frac{2\kappa+4}{\kappa+1}$. Then
$$\|\cT_\la\|_{L^2(\bbR^{d})\to L^{q_1,2}(\bbR^d)}\lc
\la^{-d/q_1}.
$$
\end{theorem}
The $L^2\to L^{q_1}$ bounds are in \cite{grse}. Given  the preparations in that work the proof of  Theorem
\ref{oscfoldthm} is
very similar  to the proof of  Theorem \ref{oscthm}. We sketch the
argument in  \S\ref{oscintfoldsect}.

\medskip

{\it Remarks.}
(i)  In two dimensions, under the stronger hypothesis of {\it two-sided}
fold singularities, together with the appropriate
curvature assumptions,  such estimates can be derived from the  sharp
$L^p \to L^q$ results for $q=3p'/2$ and  $q>5/2$, obtained   by
Bennett and the second author in \cite{BeSe}. The above mentioned example by
Bourgain suggests
that  higher dimensional   analogues of those estimates will not hold in
the full generality of our setup here.

(ii) From Theorem \ref{oscfoldthm} one can obtain
an $L^2_{\text{comp}}\to L^{q,2}_{\text{loc}}$ improvement
for
Fourier  integral operators with fold singularities
in  \cite{grse}, using arguments in that paper.
In particular this covers the  $L^2\to L^{3,2}$ and $L^{3/2,2}\to L^2$
estimates for translation invariant averages over curves in $\Bbb R^3$,
 with nonvanishing
curvature and torsion. The corresponding Lebesgue space estimates had been
already obtained by Oberlin \cite{ob} and his paper was the starting point for the  variable results
in \cite{grse}.
The   version of Theorem \ref{oscfoldthm}
for one-sided folds, in its adjoint formulation,
 also implies  the optimal
$L^{3/2,2}_{\text{comp}}(\bbR^3)\to L^{2}_{\text{loc}}(\bbR^3)$ estimate for
the restricted
X-ray transform associated to {\it well-curved} line complexes in
$\bbR^3$, see \cite{grse} for further discussion.

(iii) As observed in Appendix I  of  \cite{BeSe} the Lorentz-space improvement of the abovementioned
result by Oberlin can also be obtained by interpolation from better  $L^p\to L^q$
 bounds for $p>2$ for oscillatory integrals.
It is presently unknown whether  such better  results
hold for just one-sided folds and  suitable curvature assumptions, even in $\bbR^3$.

(iv) Theorem \ref{oscfoldthm} can be used to slightly improve
estimates for eigenfunctions of the  Laplace-Beltrami operator $\Delta$ on a
compact manifold $M$, when restricted to hypersurfaces,
see Burq, G\'erard, Tzvetkov \cite{bgt} and Hu \cite{hu}. Assume that
$e_\la$ is an eigenfunction  for $\Delta$ satisfying
$\Delta e_\la=-\la^2 e_\la$. It is proved in  \cite{hu} that  for
any hypersurface $S\subset M$, the quotient
$\|e_\la\|_{L^q(S)}/\|e_\la\|_{L^2(M)}$ is $O(\la^{(d-1)(1/2-1/q)})$
for $\frac{2d}{d-1}\le q\le \infty$. Note that
$(d-1)(1/2-1/q)=1/q$ for the endpoint $q=2d/(d-1)$. Hu's result  is based on an application of Theorem 2.2 in \cite{grse}.
The improved estimate
$$\|e_\la\|_{L^{q,2}(S)}\lc (1+\la)^{1/q}\|e_\la\|_{L^2(M)}, \quad
q=\frac{2d}{d-1}.$$
is obtained  by  applying instead Theorem \ref{oscfoldthm} (in $d-1$ dimensions, with $\kappa=d-2$)  in her argument.

\medskip

\section{Proof of Theorem \ref{lor}}\label{restrsect}
The theorem is a consequence of the
 following convolution inequality for the Fourier
transform of $\mu$ (\cf. \cite{gu} for the case of surface measure on the sphere).

\begin{proposition}\label{offdiag}
Let $\mu$ satisfy \eqref{dimhyp}, \eqref{Fmuhyp} and define
$Tf=f*\widehat \mu$.
Let
\Be \label{pone}
\rho= \frac{(d-a+2b)(d-a+b)}{(d-a)^2+3b(d-a)+b^2}, \qquad
\sigma=\frac{d-a+2b}{b}\,.
\Ee
 Then $T$ is of restricted weak type
$(\rho,\sigma)$ and of restricted weak type $(\sigma', \rho')$,
both with operator norm $O(A^{\frac{b}{d-a+b}} B^{\frac{d-a}{d-a+b}})$.
Moreover, if $\rho<p<\sigma'$ and $\frac 1p-\frac 1q
= \frac{d-a}{d-a+b}$, then  for any $s\in (0,\infty]$,
\Be\label{poneinterpol}
\|f*\widehat \mu\|_{L^{q,s}} \le \cC(p,s)
A^{\frac{b}{d-a+b}} B^{\frac{d-a}{d-a+b}} \|f\|_{L^{p,s}}.
\Ee
In particular \eqref{poneinterpol} holds for $p=p_\circ(a,b)$,
$q= (p_\circ(a,b))'$.\end{proposition}

\begin{proof}[Proof of Theorem \ref{lor}]
 Taking Proposition \eqref{offdiag} for granted the conclusion
follows from \eqref{poneinterpol} for $s=2$, $p=p_\circ$. Indeed,  by  Tomas' $T^*T$ argument, with $\widetilde f:=f(-\cdot)$,
\begin{multline}\label{tomas}\int |\widehat f(\xi)|^2 d\mu =
\int \overline f(x)\, \widetilde f*\widehat
\mu(-x) \,dx \le \|f\|_{L^{p_\circ,2}}
\|\widetilde f*\widehat \mu\|_{L^{p_\circ',2}}  \\
\lc A^{\frac{b}{d-a+b}} B^{\frac{d-a}{d-a+b}} \|f\|_{L^{p_\circ,2}}^2 .
\end{multline}
\end{proof}

\begin{remark}[{\it Bourgain's interpolation argument}]
In the  proof of Proposition \ref{offdiag} we use a trick
introduced by Bourgain \cite{bo1} in his proof of an endpoint
bound for the spherical maximal function,
see  \S 6.2 in \cite{csww} for an abstract analogue.
In this version we are given
pairs of  spaces
$\overline A=(A_0,A_1)$, $\overline B=(B_0,B_1)$, and operators $T_j$
that map $A_i$ to $B_i$ and we assume that
$\|T_j\|_{A_0\to B_0} \le M_0 2^{-j\beta_0}$ and
$\|T_j\|_{A_1\to B_1} \le M_1 2^{j\beta_1}$ for some
 $\beta_0>0$, $\beta_1>0$.
Let
$\vartheta=\frac{\beta_0}{\beta_0+\beta_1}$.
Then the result is that $T=\sum T_j$ maps the Lions-Peetre
 interpolation space
$\overline A_{\vth,1}$ to $\overline B_{\vth,\infty}$ with operator
 norm $C(\beta_0, \beta_1) M_0^{1-\vth} M_1^\vth$. In applications
we are mostly dealing with Lebesgue or Lorentz spaces and the result
then
involves a restricted weak type estimate, as in \cite{bo1}.
\end{remark}

\begin{proof}[Proof of Proposition
 \ref{offdiag}]
We prove the $L^{\rho,1}\to L^{\sigma, \infty}$
inequality.
We use the Tomas approach in \cite{tomas1}, \cite{mock} and  dyadically decompose $\widehat \mu$. Let $\chi_0$ be smooth and
supported in $\{x:|x|<1\} $ and let $\chi_0(x)=1$ for $|x|\le 1/2$.
For $j\ge 1$ let $\chi_j(x)=
\chi_0(2^{-j}x)-\chi_0(2^{1-j}x)$, so $\sum_{j=0}^\infty
\chi_j(x)\equiv 1$. Let $\mu_j=\mu* \cF^{-1}[\chi_j]$.
Since $\mu$ is a probability measure it is clear that
$\|\mu_0\|_\infty +\|\widehat \mu_0\|_\infty \lc 1 .$
As $A,B\ge 1$ it is easily verified (for  details {\it cf.} \cite{mock}) that for $j\ge 0$   assumption
\eqref{Fmuhyp}
implies
\Be \label{Fmujinfty}
\|\widehat \mu_j\|_\infty \lc B 2^{-jb}\Ee
 and that  assumption
\eqref{dimhyp} implies
\Be\label{mujinfty}
\|\mu_j\|_\infty\lc A 2^{j(d-a)}\,.\Ee
Therefore, if we define $T_j f=f*\widehat \mu_j$, we have
$\|T_j\|_{L^1\to L^\infty} \lc B 2^{-jb}$ and
$\|T_j\|_{L^2\to L^2} \lc A2^{j(d-a)}$.

Now let
\Be \label{thetadef} \theta= \frac{d-a}{d-a+b}\Ee
so that $(1-\theta)(d-a)+ \theta(-b)=0$.
 We calculate that for $p_\circ=
\frac{2(d-a+b)}{2(d-a)+b}$ we have
$(1-\theta) (\frac 12,\frac 12)+ \theta({1},
\frac{1}{\infty})=(\frac 1{p_\circ}, 1-\frac 1{p_\circ})$.
Now the two inequalities for $T_j$   allow us to apply
Bourgain's  interpolation trick;
the result is that
 the operator $T=\sum
T_j$ is of restricted weak type $(p_\circ, p_\circ')$, with operator
norm
$\le C A^{1-\theta} B^{\theta}
$, and if $I$ is a compact subinterval of $(0,\infty)$ then for $a,b\in
I$ the
constants  $C(a,b)$ depend only on $I$.
Thus we have proved
$$\|f*\widehat  \mu\|_{L^{p_\circ',\infty}}\lc A^{\frac{b}{d-a+b}} B^{\frac{d-a}{d-a+b}}
  \|f\|_{L^{p_\circ,1}}.$$
By applying Tomas' argument   we get
\begin{multline}\label{tomasweak}\int |\widehat f(\xi)|^2
  d\mu
 =\int \overline f(x) \widetilde f*\widehat
\mu(-x) dx \\ \lc \|f\|_{L^{p_\circ,1}}
\|\widetilde f*\widehat \mu\|_{L^{p_\circ',\infty}}  \lc
A^{\frac{b}{d-a+b}} B^{\frac{d-a}{d-a+b}}
 \|f\|_{L^{p_\circ,1}}^2
\end{multline}
which is weaker than \eqref{tomas}.

We use \eqref{tomasweak} to bound
$\|f*\widehat \mu_j\|_2$.
By Plancherel's theorem and \eqref{mujinfty},
\begin{align*}
\|f*\widehat \mu_j \|_2 &= \Big(\int|\widehat f(\xi)|^2 |\mu_j(\xi)|^2
d\xi\Big)^{1/2}
\\&\lc A^{1/2} 2^{j \frac{d-a}{2}}
\Big(\int|\widehat f(\xi)|^2 |\mu_j(\xi)|
d\xi\Big)^{1/2}\\
&\lc A^{1/2} 2^{j \frac{d-a}{2}}
\Big(\int |\cF^{-1}[\chi_j](\xi)|
\int|\widehat f(\eta+\xi)|^2 d\mu(\eta)
\, d\xi\Big)^{1/2}\,.
\end{align*} By \eqref{tomasweak}, this is
$$\lc A^{1/2} 2^{j \frac{d-a}{2}}
\Big(\int \frac{2^{jd}}{(1+2^j|\xi|)^{d+1}}
A^{1-\theta} B^{\theta}
\big\| f e^{\inn{2\pi i \cdot}{\xi}} \big\|_{L^{p_\circ,1}}^2
\, d\xi\Big)^{1/2}\,
$$
and hence we obtain
\Be\label{Tjpcirc2}
\|f*\widehat\mu_j\|_2 \lc
A^{1-\theta/2}
B^{\theta/2}
 2^{j \frac{d-a}{2}}
 \|f\|_{L^{p_\circ,1}}.
\Ee
We interpolate this estimate with the $L^1\to L^\infty$ bound $O(B2^{-jb})$.
Let \Be\label{gadef}
\ga= \frac{d-a}{d-a+2b}.
\Ee
so that  $(1- \ga)\frac{d-a}2 +\ga(-b)=0$. A calculation shows that if  $\rho, \sigma$ are in \eqref{pone},
then $(1-\ga)(\frac{1}{p_\circ},\frac 12)+\ga (1,\frac
1{\infty})=(\frac {1}{\rho}, \frac{1}{\sigma})$.
Thus, again by Bourgain's interpolation trick, the operator of
 convolution with $\widehat \mu$ is of restricted weak type
 $(\rho,\sigma)$, with operator norm $\lc (A^{1-\frac \theta   2}
B^{\frac\theta 2})^{1-\ga }B^\ga $.
One calculates from \eqref{thetadef}, \eqref{gadef} that
$(1-\ga)(1-\frac \theta 2)=\frac{b}{d-a+b}=1-\theta$ and
$(1-\ga)\frac \theta 2 +\ga=\frac{d-a}{d-a+b}=\theta$
which yields
$$
\|f*\widehat \mu \|_{L^{\sigma,\infty}} \lc A^{1-\theta} B^\theta
\|f\|_{L^{\rho,1}},
$$
as claimed. The corresponding $L^{\sigma',1}\to L^{\rho', \infty}$
inequality
follows by
duality. Finally, inequality
\eqref{poneinterpol} for $\rho<p<\sigma'$ follows by real
interpolation between the $L^{\rho,1}\to L^{\sigma, \infty}$ and the
$L^{\sigma',1}\to L^{\rho', \infty}$ inequality.
To obtain the $L^{p_\circ,s}\to L^{p_\circ',s}$ inequality note that
$(
1/{p_\circ}, 1/{p_\circ'})$ is the midpoint of the interval with
endpoints $(1/\rho,1/\sigma)$ and
$(1/\sigma', 1/{\rho'})$.
\end{proof}


\section{Proof of Theorem \ref{oscthm}}\label{oscintsect}
We may use a partition of unity and a compactness argument to reduce to the situation
that
 the  amplitude $\zeta$ has support in $\{(x,y): |x|<\eps^2, |y|<\eps^2\}$, for small  $\eps>0$. After  changes of variable in $x$ and in $y$  we may assume that
\begin{align} \label{phixprimey}
\varphi_{x'y}(0,0)&=I_{d-1}
\\
\label{phixprimeyy}
\varphi_{x'yy}(0,0)&=0
\\ \label{phixdy}
\varphi_{x_dy}(0,0)&=0
\\ \label{phixdyy}
\rank  \varphi_{x_dyy}(0,0)&\ge \kappa\,.
\end{align}
The conclusion of the Theorem is equivalent with the case $s=2$ of
$$\|T_\la T_\la^*\|_{L^{q_\circ',s}(\bbR^d)\to L^{q_\circ,s}(\bbR^d)}
\lc
\la^{-2d/q_\circ}.$$
We split the operator  $T_\la T_\la^*$.
Let $\eta_0\in C^\infty_0(\bbR)$ so that $\eta_0(s)=1$ for $|s|\le 1/2$ and $\eta_0$ supported in $(-1,1)$. For $j\ge 1$, let
$\eta_j= \eta_0(2^{-j}\cdot)-\eta_0(2^{-j+1}\cdot)$.
We set
\Be \label{bjdef}
\begin{aligned}
b_j(w,z,y)&= \zeta(w,y)\overline{\zeta(z,y)} \eta_j(\la (w_d-z_d))
\eta_0(\eps^{-1}\la 2^{-j}|w'-z'|)
\\
\widetilde b_j(w,z,y)&= \zeta(w,y)\overline{\zeta(z,y)} \eta_j(\la (w_d-z_d))
(1-\eta_0(\eps^{-1}\la 2^{-j}|w'-z'|))
\end{aligned}\Ee and let $\cS^\la_j$ be
 the operators with integral kernel
\begin{align*}
S^\la_j(w,z)=
\int b_j(w,z,y) e^{i\la (\vphi(w,y)-\vphi(z,y))} dy;
\end{align*}
also let $\widetilde S^\la_j(w,z)$ be similarly defined
with $\widetilde b_j$
 in place of $b_j$.
Then
\Be \label{TstarTdecomp}
T_\la T_\la^*= \sum_{j\ge 0} \cS^\la_j+
\sum_{j\ge 0} \widetilde \cS^\la_j.
\Ee
Note that $b_j$ is supported where $|w'-z'|\ll|w_d-z_d|$
and $|w_d-z_d|\approx 2^j\la^{-1}$.

For integration by parts arguments we analyze
\begin{multline*}\vphi_y'(w,y)-\vphi_y'(z,y)
= \\
\int_0^1 (w'-z')^{\text{\rm t}} \vphi_{x'y}(z+s(w-z)) ds +
\int_0^1 (w_d-z_d) \vphi_{x_dy}(z+s(w-z)) ds
\end{multline*}
and by \eqref{phixprimey}, \eqref{phixdy},
$\vphi_{x'y}=I_{d-1}+O(\eps^2)$,
$\vphi_{x_dy}=O(\eps^2)$.
On the support of $ \widetilde b_j$ we have $|w'-z'|\ge c\eps |w_d-z_d|$ and thus
\begin{align*}|\vphi_y'(w,y)-\vphi_y'(z,y)|  &\ge |w'-z'|- C\eps^2 |w_d-z_d|
\\
&\ge \eps |w-z| \text { for  } (w,z,y)\in \supp \widetilde b_j.
\end{align*}
Integration by parts with respect to $y$ yields
$$\Big|\sum_{j\ge 0} \widetilde S^\la_j(w,z)\Big| \le C_{\eps,N} (1+\la|w-z|)^{-N}. $$
From Schur's Lemma and subsequent interpolation with a trivial $L^\infty$
 bound we get for
$2<q<\infty$, $0<s\le\infty$
\begin{equation}\label{remainder}
\Big\|\sum_{j\ge 0} \widetilde \cS^\la_j \Big\|_{L^{q',s}\to L^{q,s}} \lc
 \la^{-2d/q}\,;
\end{equation}
moreover, by the support properties of $b_0$ and Schur's lemma
\begin{equation}\label{remainder2}
\big\|\cS^\la_0 \big\|_{L^{q',s}\to L^{q,s}} \lc
 \la^{-2d/q}.
\end{equation}
We shall use these  inequalities  for $s=2$.

The main task is to show that
\begin{equation} \label{maindiagest}
\Big\|\sum_{j>0} \cS^\la_j\Big\|_{L^{q_\circ',s}\to L^{q_\circ,s}} \lc \la^{-2d/q_\circ}, \quad
q_\circ=2+4\ka^{-1}\,.
\end{equation}

The crucial step in the proof of \eqref{maindiagest} is to establish
part (ii) in
\begin{proposition}\label{oscoffdiag}
(i) For $j>0$,
\Be\label{L1Linfty}
\big\|\cS^\la_j \Big\|_{L^{1}(\bbR^d)\to L^\infty(\Bbb R^d)}
\lc 2^{-j \ka/2}.
\Ee

(ii) For $q_\circ=2+4\ka^{-1}$, $j> 0$
\Be \label{oscoffdiagest}
 \big\|\cS^\la_j \Big\|_{L^{q_\circ',1}(\bbR^d)\to L^2(\Bbb R^d)} \,+\,
\big\|\cS^\la_j \Big\|_{L^2(\bbR^d)\to L^{q_\circ,\infty}(\bbR^d)}
\lc 2^{j/2}\la^{-d(\frac{1}{q_\circ}+\frac 12)}\,.
\Ee
\end{proposition}

\begin{proof}[Proof that Proposition \ref{oscoffdiag} implies
\eqref{maindiagest}]


We interpolate \eqref{L1Linfty} with the two inequalities in
\eqref{oscoffdiagest}. Let
$\rho=\frac{2(\ka+1)(\ka+2)}{\ka^2+6\ka+4}$, $\sigma=\frac{2\ka+2}{\ka}$
(which coincide with the definition in \eqref{pone} for the parameters
$(a,b)=(d-1,\ka/2)$). Then $(\frac{1}{\rho}, \frac{1}{\sigma})=
(\frac{1-\gamma}{p_\circ}+\gamma, \frac{1-\gamma}{2})$ with
$p_\circ=q_\circ'$ and $\gamma=(1+\ka)^{-1}$ as in
\eqref{gadef}.
We argue as in  the proof of Proposition
\ref{offdiag},  and obtain, by real interpolation  and  Bourgain's trick,
that
\Be  \label{tworhosigmaineq}
 \Big\|\sum_{j>0} \cS^\la_j\Big\|_{L^{\rho,1}\to L^{\sigma,\infty}}
+
\Big\|\sum_{j>0} \cS^\la_j\Big\|_{L^{\sigma',1}\to L^{\rho',\infty}} \lc
\la^{-d(1-\frac{1}{\rho}+\frac{1}{\sigma})} \,.
\Ee
Now  $1-\rho^{-1}+\sigma^{-1}
=2q_\circ^{-1}$
and we may interpolate the two inequalities in \eqref{tworhosigmaineq}
to deduce the assertion
\eqref{maindiagest}.
\end{proof}

\begin{proof}[Proof of  Proposition \ref{oscoffdiag}]
Note that by
\eqref{phixdyy}, \eqref{phixprimeyy} the determinant of a symmetric $\kappa\times\kappa$ minor
of
$\vphi_{x_dyy}$ is nonzero near the origin. This means that
for $|w'-z'|\le \eps |w_d-z_d|\approx 2^j\la^{-1}$,
the corresponding minor of
\begin{multline*}\la\big[\vphi_{y'y'}(w,y)-\vphi_{y'y'}(z,y)\big]
= \\
2^j \, \frac{w_d-z_d}{2^j\la^{-1}}\int_0^1
 \vphi_{x_dy'y'}(z+s(w-z)) ds +O(\eps 2^j)
\end{multline*}
has determinant  $\approx 2^j$.
Now inequality  \eqref{L1Linfty} follows easily
by a  stationary phase argument with respect to the relevant $\kappa$ coordinates.

For part (ii) we need  only prove the
$L^{2}\to L^{q_\circ,\infty} $ inequality since the
$L^{p_\circ,1}\to L^{2} $ inequality follows by taking adjoints
and replacing $\vphi$ by $-\varphi$.

We first notice  that $\cS^\la_j$ is identically zero if $2^j \gc \eps\la$ and in all other cases  $\cS^\la_j$
is essentially local on balls of diameter
$\approx 2^j\la^{-1}$. This means if $r\ge 2^j/ \la$ and  $f$ is supported in
the ball $B(a ,r)$ centered at $a\in \bbR^d$
then $\cS^\la_j f$ is supported in $B(a ,Cr)$.
Therefore it suffices to prove the inequality for functions $f$ supported in $B(a, 2^j\lambda^{-1})$.
We set \Be\label{mudelta}\mu=2^j, \qquad  \delta=2^j\la^{-1},
\Ee
and change variables
$w=a +\delta u$ and $z=a +\delta v$.

Then $$\cS_j^\la f(a +\delta v)=\delta^d \cR_{\mu}[f(a +\delta \cdot)]$$
where \Be\label{Rmudef}
\cR_\mu g(u)\equiv \cR_\mu^{a ,\delta} g(u)=
\iint e^{-i\mu \Psi(u,v,y)} \beta(u,v,y) dy \,g(v) dv,
\Ee
with
\begin{align}\label{Psidef}
\Psi(u,v,y)&\equiv \Psi(u,v,y;a ,\delta)
= \frac 1\delta\big[\vphi(a+\delta u,y)-\vphi(a+\delta v,y)\big]\\
&= \int_0^1
\inn{u-v}{\nabla_x\vphi(a +\delta(v+s(u-v)),y)} \, ds
\notag
\end{align} and
$$\beta(u,v,y)= b_j(a +\delta u,a +\delta v,y)$$
with  $\delta=2^j\la^{-1}$.  By rescaling
$$\|\cS^{\la}_{j}\|_{L^2\to L^{q_\circ,\infty}}
\lc \delta^{\frac d2+\frac d{q_\circ}}
\sup_{|a |\lc \eps^2}\|\cR_\mu^{a ,\delta}\|\ci{L^2\to L^{q_\circ,\infty}},
\qquad \delta=2^j\la^{-1}=\mu \la^{-1}
$$
and thus we just need to show that for $\mu\ge 1$
$$ \|\cR_\mu\|_{L^2\to L^{q_\circ,\infty}}\lc \mu^{-\frac{d-1}{2} - \frac d{q_{\circ}}}.$$
This  of course follows from
\Be \label{R-restr-est}
\|\cR_\mu \cR_\mu^*\|_{L^{q_\circ',1}\to L^{q_\circ,\infty}}\lc
 \mu^{1-d -  {2d}/{q_{\circ}}}\,.
\Ee
We proceed to show \eqref{R-restr-est}
using an analogue of  Tomas' interpolation argument.
The Schwartz kernel  of $\cR_\mu \cR_\mu^*$ is given by
\Be
K_\mu(u,\tu)=
\iiint e^{i\mu (\Psi(u,v,y+h)-\Psi(\tu,v,y))}\gamma (u,\tu; v,y, h) \, dv\, dy\, dh
\Ee
where $\gamma(u,\tu;v,y,h)
=\beta(u,v,y+h)\overline{\beta(\tu,v,y)}$.

We now reduce the number of frequency variables by a straightforward
stationary phase arguments.
Let
\begin{multline*}
\theta(u,\tu,v',v_d, y,h)=\Psi(u,v,y+h)-\Psi(\tu,v,y)
\\=\frac 1\delta\big[
\vphi(a+\delta u, y+h)-\vphi(a+\delta v, y+h)
-\vphi(a+\delta \tu, y)+\vphi(a+\delta v, y)\big].
\end{multline*}
Then the partial Hessian of $\theta$  with
respect to the $(v',h)$-variables is given by
\begin{align*}
\begin{pmatrix} -\delta \vphi_{x'x'}(a+\delta v, y+h)&
-\vphi_{ x'y}(a+\delta v, y+h)\\ -\vphi_{y x'}(a+\delta v, y+h)
&
\inn{u-v}{\int_0^1\varphi_{xyy}(a+ \delta(v+s(u-v))y)\,ds}\end{pmatrix}\,.
\end{align*}
It is clearly nondegenerate on the support of our cutoff functions (with small $\eps$).
We observe that
\begin{align*}
\theta_{v'}=0\quad&\iff \quad \varphi_{x'}(a+\delta v, y+h)=
\varphi_{x'}(a+\delta v, y)
\\
\theta_{h}=0\quad&\iff \quad \inn{u-v}{\int_0^1\varphi_{xy}(a+ \delta(v+s(u-v))y)\,ds} =0
\end{align*}
 and these equations are solved by $h=0$ and $v'=\fv'(u,v_d,y)$
 for some smooth $\fv'$. We now observe that when $\theta$ is evaluated at
$h=0$, the result is independent of $v$, in fact
$$\theta(u,\tu,\fv'(u,v_d,y),v_d, y,0)= \frac{1}{\delta}\big[
\vphi(a+\delta u, y)-\vphi(a+\delta \tu, y)\big] = \Psi(u,\tilde u ,y).$$
The method of stationary phase
(applied in the $(v',h)$-variables)
gives
\Be
K_\mu(u,\tu)= \mu^{1-d}
\int e^{i\mu\Psi(u,\tu,y)}\alpha (u,\tu,y)  dy
\Ee
for suitable smooth amplitudes $\alpha$ depending smoothly on the
 parameters $a$ and $\delta$.

We now decompose the kernel in a way analogous to \eqref{TstarTdecomp}.
Split coordinates in $\bbR^d$ as $u=(u',u_d)$ and let,
for $\ell=0,1,2,\dots$ (and $\eta_\ell$ as in \eqref{bjdef})
\begin{align*}
\alpha_\ell(u,\tu,y)&= \alpha(u,\tu,y)
 \eta_\ell(\mu (u_d-\tu_d))
\eta_0(\eps^{-1}
\mu 2^{-\ell}|u'-\tu'|)
\\
\widetilde\alpha_\ell(u,\tu,y)&= \alpha(u,\tu,y)
 \eta_\ell(\mu (u_d-\tu_d))
(1-\eta_0(\eps^{-1}\mu 2^{-\ell}|u'-\tu'|)).
\end{align*} Let $\cV^\mu_\ell$
denote the operators with integral kernel
\begin{align*}
V^\mu_\ell=\mu^{1-d}
\int e^{i\mu \Psi(u,\tu,y)}\alpha_\ell (u,\tu,y) \,
 dy
\end{align*}
and let $\widetilde \cV^\mu_\ell$ and the kernel $\widetilde V^\mu_\ell$  be analogously defined with
$\widetilde \alpha_\ell$ in place of
$\alpha_\ell$.
 Then
\Be \label{RstarRdecomp}\cR_\mu \cR_\mu^*=
\sum_{\ell\ge 0} \widetilde {\cV}^\mu_\ell
+\sum_{\ell\ge 0} \cV^\mu_\ell\,.
\Ee

The straightforward  argument used for
\eqref{remainder} and \eqref{remainder2}
 now yields for
$2<q<\infty$, $0<s\le\infty$,
\begin{equation}\label{remaindersecond}
\big\|\mu^{d-1}\cV^\mu_0\big\|_{L^{q',s}\to L^{q,s}} +
\Big\|\mu^{d-1}\sum_{\ell\ge 0} \widetilde \cV^\mu_\ell
\Big\|_{L^{q',s}\to L^{q,s}} \lc \mu^{-2d/q}.
\end{equation}

We need to prove the appropriate
 $L^1\to L^\infty$ and $L^2\to L^2$ bounds for $\cV^\mu_\ell$
which are
\begin{align}\label{L1Linftysecond}
&\big\|\mu^{d-1}\cV^\mu_\ell \big\|_{L^{1}(\bbR^d)\to L^\infty(\Bbb R^d)}
\lc 2^{-\ell \ka/2}\, ,
\\ \label{L2L2second}
&\big\|\mu^{d-1}\cV^\mu_\ell \big\|_{L^{2}(\bbR^d)\to L^2(\Bbb R^d)}
\lc 2^{\ell}\mu^{-d} \, .
\end{align}
Then, by Bourgain's interpolation trick,
\eqref{L1Linftysecond} and \eqref{L2L2second} imply
\eqref{R-restr-est}.

It only remains to prove
\eqref{L1Linftysecond}
 and \eqref{L2L2second}.
The inequality \eqref{L1Linfty} (written for $(\ell, \mu, \alpha_\ell)$  in place of $(j, \la, \gamma_j$)) immediately yields \eqref{L1Linftysecond}.

For the $L^2$ bound we observe that
only the cases $2^\ell\lc \eps \mu$ are relevant and
that $\cV^\mu_\ell$
is essentially local on balls of diameter
$\approx 2^\ell\mu^{-1}$. Therefore it suffices to prove the inequality
for   $f$ supported in
the ball $B(b ,r)$ with $r= 2^\ell\mu^{-1}$.
We rescale and set
$u=b+2^\ell\mu^{-1}\om$ and $\tu=b +2^\ell\mu^{-1}\tom$.
Then \Be\label{secondrescale}
\cV_\ell^\mu f(b +2^\ell\mu^{-1}\omega)=(2^\ell\mu^{-1})^d
\cW^{\ell}[f(b +2^\ell\mu^{-1} \cdot)](\omega)
\Ee
where
\Be\label{Wellmudef}
\cW^\ell h(\om)=
 \iint e^{-i2^\ell  \Theta(\om,\tom,y)} \sigma(\om,\tom,y) dy \,h(\tom) d\tom,
\Ee
with
\begin{align*}\Theta(\om,\tom,y)
&= \frac{\mu}{ 2^{\ell}\delta}
\big[\vphi(a+\delta b+ \tfrac{2^\ell\delta}{\mu} \om,y)-
\vphi(a+\delta b+ \tfrac{2^\ell\delta}{\mu} \tom,y)\big]
\\&=
\inn{\om-\tom}{\int_0^1
\vphi_x(a+\delta b+ 2^\ell\la^{-1}(\tom+s(\om-\tom)),y) ds}\,.
\end{align*}
Recall that $\mu=2^j$ and $\delta= 2^j\la^{-1}$
so that $2^\ell \delta\mu^{-1}= 2^\ell\la^{-1}\le \eps$.
The phase $\Theta$ and therefore the operator $\cW^\ell$ depend
 on the points $a,b$ but the estimates will be uniform.

 We now show
that
\Be\label{W-ell-bound}
\|\cW^\ell\|_{L^2(\bbR^d)\to L^2(\bbR^d)}\lc 2^{-\ell(d-1)}
\Ee
which, by
\eqref{secondrescale}, implies that $\|\mu^{d-1} \cV_\ell^\mu\|_{L^2} \lc (2^\ell\mu^{-1})^d 2^{-\ell(d-1)}$. This is
\eqref{L2L2second}.

Finally, \eqref{W-ell-bound} follows from a standard $T^*T$ estimate
for oscillatory integral operators associated to a canonical graph,
here with $d-1$ frequency variables $y$ and space variables $\om'$,
$\tom'$, with frozen
$\om_d$, $\tom_d$.
For this result we refer to  Lemma 2.3 in \cite{grse}
which is built on an argument in   \cite{Hfio}. The required estimate
follows after noting that
$$\begin{pmatrix} \Theta_{\om'\tom'}&\Theta_{\om'y}
\\ \Theta_{y\tom'}&\Theta_{yy}\end{pmatrix}
= \begin{pmatrix} 0 &\vphi_{x'y}
\\ \vphi_{yx'}&\inn{\om-\tom}{\vphi_{xyy}}\end{pmatrix} \Big|_{(0,0)}
+O(\eps)
$$
has determinant bounded away from $0$. This  is immediate from
\eqref{phixprimey} (provided that $\eps$ is chosen small).
\end{proof}

\section{Proof of Theorem \ref{oscfoldthm}}\label{oscintfoldsect}
The proof is quite  analogous to
the proof of Theorem \ref{oscthm}, and therefore we will give only a sketch.
As discussed in \cite{grse} one can assume after suitable changes of variables in the $x$ and the $y$ coordinates that,
with the $y$-variables  split as $y=(y', y_d)$,
\Be \label{rankpiL}
\begin{pmatrix} \Phi_{x'y'}&\Phi_{x'y_d}\\
\Phi_{x_dy'}&\Phi_{x_dy_d} \end{pmatrix}\Big|_{(0,0)}=
\begin{pmatrix} I_{d-1}&0\\0&0\end{pmatrix}
\Ee
and\begin{align} \label{foldpiL}
 \Phi_{x_dy_dy_d}(0,0)&\neq 0\, ,
 \\ \label{Phixdyprimeyd}
 \Phi_{x_dy'y_d}(0,0)&=0 \,.\end{align}
\eqref {foldpiL} reflects the fold condition on $\pi_L$.
Moreover,
\Be
\label{curvpiL}
\rank (\Phi_{x_d y'y'})\ge \kappa
\Ee
which expresses the curvature condition.
Note that  by
\eqref{foldpiL}, \eqref{Phixdyprimeyd} and \eqref{curvpiL},
\Be \label{rankkaplusone}
\rank (\Phi_{x_d yy})\ge \kappa+1.
\Ee


We shall argue  as in the proof of Theorem \ref{oscthm} and show that
$$\|\cT_\la \cT_\la^*\|_{L^{q_1',s}(\bbR^d) \to L^{q_1,s}(\bbR^d)} \le \la^{-2d/q_1}.$$
We proceed splitting the operator $\cT_\la \cT_\la^*$ as in \eqref{TstarTdecomp}
with the only difference that $\vphi$ is replaced with $\Phi$
and now the $y$ integrations are over small open sets in $\bbR^d$.
The proof of the estimate analogous to \eqref{remainder} is exactly
 the same, and then again the main task is
to establish that
\begin{equation} \label{mainfolddiagest}
\Big\|\sum_{j>0} \cS^\la_j\Big\|_{L^{q_1',s}\to L^{q_1,s}} \lc \la^{-2d/q_1}, \quad
q_1=\frac{2\ka+4}{\ka+1}\,.
\end{equation}
The following estimates are analogous to Proposition \ref{oscoffdiag}:
\Be\label{L1Linftyfold}
\big\|\cS^\la_j \Big\|_{L^{1}(\bbR^d)\to L^\infty(\Bbb R^d)}
\lc 2^{-j (\ka+1)/2},
\Ee
and
\Be \label{oscoffdiagfoldest}
 \big\|\cS^\la_j \Big\|_{L^{q_1',1}(\bbR^d)\to L^2(\Bbb R^d)} \,+\,
\big\|\cS^\la_j \Big\|_{L^2(\bbR^d)\to L^{q_1,\infty}(\bbR^d)}
\lc 2^{j/2}\la^{-(d/q_1+d/2)}\,.
\Ee

Given \eqref{L1Linftyfold} and \eqref{oscoffdiagfoldest}
Bourgain's interpolation argument  shows that
\Be  \label{tworhosigmaineqfold}
 \Big\|\sum_{j>0} \cS^\la_j\Big\|_{L^{\rho_1,1}\to L^{\sigma_1,\infty}}
+
\Big\|\sum_{j>0} \cS^\la_j\Big\|_{L^{\sigma_1',1}\to L^{\rho_1',\infty}} \lc
\la^{-d(1-\frac{1}{\rho_1}+\frac{1}{\sigma_1})}
\Ee
where
$\rho_1$, $\sigma_1$ are  as in \eqref{pone}, with
$a=d-1$, $b=(\ka+1)/2$.
 Then $(\frac{1}{\rho_1}, \frac{1}{\sigma_1})=
(\frac{1-\gamma}{q_1'}+\gamma, \frac{1-\gamma}{2})$ with
$\gamma=(2+\ka)^{-1}$.
Since   $1-\rho_1^{-1}+\sigma_1^{-1}
=2q_1^{-1}$
we get
\eqref{mainfolddiagest}.

It remains to establish the estimates
\eqref{L1Linftyfold} and \eqref{oscoffdiagfoldest}.
Again, \eqref{L1Linftyfold} follows using the method of stationary
phase and the better bound is due to the condition
 \eqref{rankkaplusone}.
The estimate  \eqref{oscoffdiagfoldest}
is proved analogously to
\eqref{oscoffdiagest} above. The phase functions $\Psi$, $\theta$,
$\Theta$ as well as the operators
$\cR_\mu$, $\cV^\mu_\ell$, $\widetilde \cV^\mu_\ell$ and then
$\cW^\ell$ are defined as in the proof of Proposition \ref{oscoffdiag},
with the exception that all $y$ integrations are over a small  open
set  in $\Bbb R^d$ (instead $\Bbb R^{d-1}$ above).
We need to show  the analogue of \eqref{R-restr-est}, namely that the
$L^{q_1',1}\to L^{q_1,\infty} $ operator norm of $\cR_\mu \cR_\mu^*$
is $O(\mu^{1-d-2d/q_1})$. As above this follows from the analogues
of \eqref{L1Linftysecond}, \eqref{L2L2second} which read now
\begin{align}\label{L1Linftysecondfold}
&\big\|\mu^{d-1}\cV^\mu_\ell \Big\|_{L^{1}(\bbR^d)\to L^\infty(\Bbb R^d)}
\lc 2^{-\ell (\ka+1)/2}\, ,
\\ \label{L2L2secondfold}
&\big\|\mu^{d-1}\cV^\mu_\ell \Big\|_{L^{2}(\bbR^d)\to L^2(\Bbb R^d)}
\lc 2^{\ell/2}\mu^{-d}\, .
\end{align}
The stationary phase argument using
 \eqref{rankkaplusone}
implies \eqref{L1Linftysecondfold}.

Turning to \eqref{L2L2secondfold} the previous  rescaling argument
reduces matters to a better estimate for the  $\cW^\ell$,
namely the bound
\Be\label{W-ell-boundfold}
\|\cW^\ell\|_{L^2(\bbR^d)\to L^2(\bbR^d)}\lc 2^{-\ell(d-\frac12)}\,.
\Ee
The gain of a factor of
$2^{-\ell/2}$ compared to \eqref{W-ell-bound}
(already crucial  in \cite{grse}) comes from the fact
that we
 now have $d$ frequency
variables $y$ and that we can use the fold condition for $\pi_L$.
We freeze  $\om_d$, $\tom_d$ and observe that in the domain of
$\Theta$
\Be \label{omlocalization}
 |\om'-\tom'|\ll |\om_d-\tom_d|\approx
1.\Ee
As before we check the standard condition
for a phase function parametrizing a canonical graph for the phase
$(\om', \tom',y)\mapsto \Theta(\om, \tom, y)$. That is, the
determinant of
\begin{multline*}\!\!\!\!\!\begin{pmatrix} \Theta_{\om'\tom'}&\Theta_{\om'y'}& \Theta_{\om'y_d}
\\ \Theta_{y'\tom'}&\Theta_{y'y'}&\Theta_{y'y_d}
\\ \Theta_{y_d\tom'}&\Theta_{y_dy'}&\Theta_{y_dy_d}
\end{pmatrix}
=
\begin{pmatrix} 0 &\Phi_{x'y'}&\Phi_{x'y_d}
\\ \Phi_{y'x'}&\inn{\om-\tom}{\Phi_{xy'y'}}
&\inn{\om-\tom}{\Phi_{xy'y_d}}
\\ \Phi_{y_dx'}&\inn{\om-\tom}{\Phi_{xy_dy'}}
&\inn{\om-\tom}{\Phi_{xy_dy_d}}
\end{pmatrix} \bigg|_{(0,0)}\\
+O(\eps)
\end{multline*}
is bounded away from $0$.
It is easily seen from
\eqref{rankpiL}, \eqref{Phixdyprimeyd}
that this determinant is equal to
$\inn{\om-\tom}{\Phi_{xy_dy_d}}+O(\eps)$ and by
\eqref{omlocalization} this is
equal to $(\om_d-\tom_d)\Phi_{x_dy_dy_d}+O(\eps)$. Thus the
canonical graph condition is satisfied by the fold condition
\eqref{foldpiL} and
\eqref{W-ell-boundfold} follows
from  Lemma 2.3 in \cite{grse}. \qed

\end{document}